\documentclass[11pt]{article}

\usepackage{amsmath,amssymb,amsthm,mathtools}
\usepackage[authoryear]{natbib}
\usepackage[colorlinks,citecolor=blue,urlcolor=blue]{hyperref}
\usepackage{enumitem}
\usepackage{bm}
\usepackage{tikz}
\usetikzlibrary{decorations.pathreplacing}
\usepackage{mathrsfs}
\usepackage{tcolorbox}
\usepackage[margin=2.5cm]{geometry}

\theoremstyle{plain}
\newtheorem{theorem}{Theorem}[section]
\newtheorem{proposition}[theorem]{Proposition}

\theoremstyle{definition}

\newtheorem{example}[theorem]{Example}

\newtheorem{procedure}[theorem]{Procedure}

\newtheorem{openproblem}[theorem]{Open problem}

\newcommand{\R}{\mathbb{R}}

\newcommand{\E}{\mathbb{E}}
\newcommand{\Var}{\mathrm{Var}}

\newcommand{\KL}{\mathrm{KL}}
\newcommand{\one}{\mathbf{1}}
\newcommand{\cM}{\mathcal{M}}
\newcommand{\cF}{\mathcal{F}}
\newcommand{\cD}{\mathcal{D}}
\newcommand{\cS}{\mathcal{S}}

\newcommand{\cE}{\mathcal{E}}

\newcommand{\Hd}{\{0,1\}^d}

\begin{document}

\title{Influence as soft sparsity: Estimation of monotone
  functions on $\{0,1\}^d$}
\author{G\'erard Biau\\
  \small Sorbonne Universit\'e and Institut Universitaire de France\\
  \small\texttt{gerard.biau@sorbonne-universite.fr}}
\date{}
\maketitle

\begin{abstract}
We study the problem of estimating a monotone function
  $f:\{0,1\}^d\to[0,1]$ from noisy observations at uniformly
  random vertices of the Boolean hypercube.  As a measure of
  complexity for the target~$f$, we use the total
  $L^1$-influence
  $I(f)=\sum_{i=1}^d(\E[f(X)\mid X_i=1]-\E[f(X)\mid X_i=0])$,
  a classical quantity in Boolean analysis that is nonnegative
  for monotone functions and controls the effective
  dimensionality of the estimation problem: through a spectral
  concentration result in the spirit of Friedgut's junta
  theorem, the Fourier spectrum of any $f$ with $I(f)\leqslant K$
  concentrates on low-degree subsets of the influential
  coordinates.  We establish minimax bounds over the class
  $\cF_K=\{f:\{0,1\}^d\to[0,1],\;f\text{ monotone},\;
  I(f)\leqslant K\}$:
  \[
    c\,\frac{K^2}{(\log n)^{3/2}}
    \;\leqslant\;
    \inf_{\hat f}\;\sup_{f\in\cF_K}\;
    \E\bigl[\|\hat f - f\|_2^2\bigr]
    \;\leqslant\; C\,\frac{K}{\sqrt{\log n}},
  \]
  where $n$ is the sample size.  The upper bound holds for all
  $K\geqslant 1$ and is uniform in the ambient dimension~$d$
  (under the mild condition $\log d\leqslant n^{1-\varepsilon}$).
  It is achieved by a Fourier thresholding estimator that adapts
  to the unknown~$K$.  The lower bound relies on a
  Varshamov--Gilbert packing on the middle layer of the
  hypercube combined with Fano's inequality.
\end{abstract}

\section{Introduction}\label{sec:intro}

\subsection{Context and motivation}\label{ssec:context}

Binary features arise naturally in many applied settings.  In clinical
medicine, patient profiles are described by the presence or absence of
risk factors---smoking, hypertension, diabetes, family history of
disease---and the probability of an adverse outcome is naturally
modelled as a \emph{monotone} function of these binary indicators,
encoding the prior that each additional risk factor can only increase
the risk \citep{feelders2010}.  In credit scoring, a borrower is
characterized by binary flags (employed or not, homeowner or not,
previous default or not) and the default probability is monotone in
these covariates \citep{ben-david1995}.  In pharmacogenomics, the
response to a drug may depend monotonically on the presence or absence
of deleterious genetic variants \citep{weinreich2013}.

In all these examples the number~$d$ of binary features can be large
(tens to thousands), while the sample size~$n$ may be comparatively
modest.  The monotonicity constraint alone does not overcome the
curse of dimensionality: even in the simplest case of Boolean-valued
functions, the number of monotone $f:\{0,1\}^d\to\{0,1\}$ grows
faster than any exponential in~$d$ (it is the Dedekind
number~$D(d)$, satisfying
$\log_2 D(d)\sim\binom{d}{\lfloor d/2\rfloor}\approx
2^d/\sqrt{d}$), and estimation over the full class~$\cM_d$ of
monotone functions $\{0,1\}^d\to[0,1]$ is at least as hard.
In practice, however, domain experts often believe that only a
moderate number of features meaningfully affect the outcome---but
\emph{which} features is unknown.

This motivates the search for a complexity measure that captures
the effective number of relevant features, is compatible with the
monotonicity constraint, and leads to a tractable estimation problem
when it is small.  A natural candidate is the \emph{total $L^1$-influence} of~$f$,
a quantity studied in the analysis of Boolean functions
\citep{kahn1988,odonnell2014,kelman2021} that, as we show, is
particularly well suited to the statistical setting.

\subsection{The model}\label{ssec:model}
Let $d\geqslant 1$.  We work on the Boolean hypercube
$\Hd$, equipped with the coordinatewise partial order
($x\leqslant y$ iff $x_i\leqslant y_i$ for all~$i$) and the uniform measure
$\mu=\mathrm{Unif}(\Hd)$.  We observe $n$ independent pairs
\begin{equation*}
  Y_j = f(X_j) + \varepsilon_j, \qquad j=1,\ldots,n,
\end{equation*}
where $X_1,\ldots,X_n\stackrel{\mathrm{iid}}{\sim}\mu$,
$\varepsilon_1,\ldots,\varepsilon_n$ are independent centered
random variables satisfying the sub-Gaussian tail condition
$\E[e^{t\varepsilon_j}]\leqslant e^{\sigma^2 t^2/2}$ for all
$t\in\R$, and the $\varepsilon_j$ are independent of the $X_j$. 

The function $f:\Hd\to[0,1]$ is assumed to be
\emph{monotone} (coordinatewise nondecreasing):
\begin{equation*}
  x\leqslant y \;\Longrightarrow\; f(x)\leqslant f(y).
\end{equation*}
We denote by $\cM_d$ the class of all monotone functions
$\Hd\to[0,1]$, and by
\[
  \cD_n=\{(X_1,Y_1),\ldots,(X_n,Y_n)\}
\]
the observed sample.
The performance of an estimator
$\hat f=\hat f(\cD_n)$ is measured by the integrated squared risk
\begin{equation*}
  R(\hat f,f)
  := \E_{\cD_n}\bigl[\|\hat f - f\|_2^2\bigr],
  \qquad
  \|g\|_2^2 := \E_\mu[g(X)^2]
  = \frac{1}{2^d}\sum_{x\in\Hd}g(x)^2,
\end{equation*}
where $\E_{\cD_n}$ denotes expectation over the sample
and $\E_\mu$ denotes expectation over $X\sim\mu$.
When there is no ambiguity, we simply write~$\E$.

For $f:\Hd\to\R$, $i\in[d]$, and $x\in\Hd$, write
$x^{i\to b}$ for~$x$ with the $i$-th coordinate replaced by~$b$,
and $\Delta_i f(x):=f(x^{i\to 1})-f(x^{i\to 0})$ for the discrete
derivative.  The \emph{$L^2$-influence} of coordinate~$i$ on~$f$ is
$I_i^{(2)}(f) := \E[(\Delta_i f(X))^2]$, and the \emph{total
$L^2$-influence} is $I^{(2)}(f):=\sum_{i=1}^d I_i^{(2)}(f)$.  This
is the classical notion of influence introduced by \citet{kahn1988};
see \citet[Chapter~2]{odonnell2014} for a thorough treatment.

When $f:\Hd\to[0,1]$ is monotone, the discrete derivative
$\Delta_i f$ takes values in~$[0,1]$, and a second notion becomes
natural: the \emph{$L^1$-influence}
$I_i(f) := \E[|\Delta_i f(X)|] = \E[\Delta_i f(X)]$, with total
$I(f):=\sum_{i=1}^d I_i(f)$. The pointwise bound
$(\Delta_i f)^2\leqslant\Delta_i f$ (valid since
$\Delta_i f\in[0,1]$) gives
\begin{equation}\label{eq:L1-L2}
  I_i(f)^2 \;\leqslant\; I_i^{(2)}(f) \;\leqslant\; I_i(f)
  \qquad\text{for all }i\in[d],
\end{equation}
where the left inequality is Jensen's.  Summing, we see that
$I^{(2)}(f)\leqslant I(f)$.

The reason we work with the $L^1$-influence rather than the
$L^2$-influence is the identity
\begin{equation}\label{eq:inf-means}
  I_i(f) = \E[f(X)\mid X_i=1] - \E[f(X)\mid X_i=0],
\end{equation}
which holds because $X^{i\to b}\stackrel{d}{=}(X\mid X_i=b)$
under the uniform measure.  Each $I_i(f)$ is thus a simple difference of conditional means,
directly estimable from data at parametric rate.  At the same time,
the bound $I^{(2)}(f)\leqslant I(f)$ ensures that controlling the
$L^1$-influence is enough to obtain Fourier-analytic approximation
guarantees.  It is this combination---statistical tractability and
analytical power---that makes the $L^1$-influence the right
complexity measure for our problem.

With this in mind, we study the class
\begin{equation*}
  \cF_K := \bigl\{f:\Hd\to[0,1] : f\in\cM_d,\; I(f)\leqslant K\bigr\}
\end{equation*}
for a parameter $K>0$.  Since $I_i(f)\leqslant 1$ for each~$i$,
we always have $I(f)\leqslant d$, so $\cF_d=\cM_d$ and the constraint
$I(f)\leqslant K$ is informative when $K<d$.

The following examples, borrowed from the analysis of Boolean
functions (see \citealp{odonnell2014}), illustrate the
range of behaviors captured by the total influence.

\begin{example}\label{ex:examples}
  \emph{(a)~Dictator.}\; The output copies a single coordinate:
  $f(x)=x_1$.  Only one variable matters, and $I(f)=1$.

  \emph{(b)~Additive junta.}\; The output averages $s$ coordinates:
  $f(x)=(1/s)\sum_{i=1}^s x_i$ for $s\leqslant d$.  The
  function depends on~$s$ coordinates, yet its total influence
  is~$1$, independent of~$s$.

 \emph{(c)~Tribes.}\; The coordinates are divided into $\ell$ blocks
of size~$w\approx\log_2 d$, and the output is~$1$ if and only if
at least one block is entirely composed of~$1$'s:
$f(x)=\bigvee_{j=1}^\ell\bigwedge_{i\in T_j}x_i$.  This function
is monotone, depends on all~$d$ coordinates, and satisfies
$I(f)=\Theta(\log d)$.

  \emph{(d)~Majority.}\; The output is~$1$ if more than half the
  coordinates are~$1$: $f(x)=\one\{\sum_i x_i > d/2\}$.  Every coordinate contributes equally, and
  $I(f)=\Theta(\sqrt{d})$.
\end{example}

These examples cover the full range $K\in\{1,\Theta(\log d),
\Theta(\sqrt{d})\}$, showing that the parameter~$K$ need not be a
universal constant: it may depend on~$d$, and our results are
stated for all $K\leqslant d$.

\subsection{Main contributions}\label{ssec:contributions}

We establish minimax bounds on the estimation risk over~$\cF_K$
that reveal a sharp dependence on the influence budget~$K$ and
the sample size~$n$, uniform in the ambient dimension~$d$.  The
upper bound $O(K/\sqrt{\log n})$ is achieved by a Fourier
thresholding estimator that adapts to the unknown~$K$ and remains
valid for~$d$ growing nearly exponentially in~$n$.  The lower
bound $\Omega(K^2/(\log n)^{3/2})$ is proved by a
Varshamov--Gilbert packing on the middle layer of the Boolean
hypercube, combined with Fano's inequality.  A finer analysis,
stratifying~$\cF_K$ by the $L^2$-influence $I^{(2)}(f)$, reveals
that the apparent $\log n$ gap between the two bounds is an
aggregation artifact: on each sub-class
$\cF_{K,B}:=\{f\in\cF_K: I^{(2)}(f)\leqslant B\}$, the gap
reduces to~$\sqrt{\log n}$.  Precise statements are given in
Section~\ref{sec:main}.

\section{Main results}\label{sec:main}

\subsection{Minimax bounds}\label{ssec:minimax}
Our goal is to determine how the worst-case estimation risk over
$\cF_K$ depends on the sample size~$n$ and the complexity
parameter~$K$.  The following theorem provides matching upper and
lower bounds, up to a gap discussed below.

\begin{theorem}[Minimax bounds]\label{thm:main}
  For every $\varepsilon\in(0,1)$, there exist constants $c,C>0$
  (depending only on $\sigma$ and~$\varepsilon$) such that the
  following holds for all integers $n\geqslant 2$ and
  $d\geqslant 1$.
  \begin{enumerate}[label=(\roman*)]
  \item \textbf{Upper bound.}\; For every
    $1\leqslant K\leqslant d$, provided
    $\log d\leqslant n^{1-\varepsilon}$, there exists an
    estimator $\hat f_n$ (depending on~$\cD_n$ but not
    on~$K$) such that
    \begin{equation}\label{eq:upper}
      \sup_{f\in\cF_K}R(\hat f_n,f)
      \;\leqslant\; C\,\frac{K}{\sqrt{\log n}}.
    \end{equation}
  \item \textbf{Lower bound.}\; For every
    $K\leqslant c\sqrt{\log n}$ and $d\geqslant(1/c)\log n$,
    \begin{equation*}
      \inf_{\tilde f}\;\sup_{f\in\cF_K}R(\tilde f,f)
      \;\geqslant\; c\,\frac{K^2}{(\log n)^{3/2}},
    \end{equation*}
    where the infimum is over all estimators
    $\tilde f=\tilde f(\cD_n)$.
  \end{enumerate}
\end{theorem}

This result calls for several comments.  The condition $K\geqslant 1$
is a normalization: for very small~$K$, the rate $K/\sqrt{\log n}$
vanishes and the estimation problem becomes degenerate.  The
requirement $\log d\leqslant n^{1-\varepsilon}$ is extremely mild:
for any fixed $\varepsilon\in(0,1)$, it allows~$d$ as large as
$\exp(n^{1-\varepsilon})$, far beyond any practical scenario.  The
estimator $\hat f_n$ does not depend on~$K$ and
achieves~\eqref{eq:upper} simultaneously for all
$1\leqslant K\leqslant d$, without prior knowledge of the influence
budget.  The lower bound is stated for $K\leqslant c\sqrt{\log n}$;
beyond this regime, the upper bound $CK/\sqrt{\log n}$ exceeds a
constant, and the estimation problem is inherently limited by the
size of the class~$\cF_K$.  The condition $d\geqslant(1/c)\log n$
is a mild dimensionality requirement: the lower bound construction
uses $\Theta(\log n)$ coordinates, so $d$ must be at least of this
order for the packing to be feasible.  

The rate $K/\sqrt{\log n}$ depends on~$d$ only through the mild
condition $\log d\leqslant n^{1-\varepsilon}$, and is otherwise
dimension-free.  This stands in contrast with the classical theory
of multivariate isotonic regression on~$[0,1]^d$, where
\citet{han2019} showed that the minimax rate is~$n^{-1/d}$ (up to
logarithmic factors)---a rate that becomes uninformative for
$d\gtrsim\log n$.  The influence constraint thus provides a
mechanism for meaningful estimation even when $d$ grows nearly
exponentially in~$n$.

The constant estimator $\hat f\equiv\bar Y$ achieves
$R(\bar Y,f)\leqslant\Var(f)+O(\sigma^2/n)\leqslant
K/4+O(\sigma^2/n)$, where $\Var(f):=\E_\mu[(f(X)-\E_\mu f(X))^2]$
denotes the variance of~$f$ under~$\mu$, and we used
$\Var(f)\leqslant I(f)/4$
(see, e.g., \citealt[Chapter~2]{odonnell2014}).  This
risk is bounded away from zero for fixed~$K$: the dictator
$f(x)=x_1\in\cF_1$ has $\Var(f)=1/4$, so
$R(\bar Y,f)\geqslant 1/4$ for all~$n$.  In contrast, our
estimator $\hat f_n$ achieves $R(\hat f_n,f)\leqslant
CK/\sqrt{\log n}$, which tends to zero as $n\to\infty$ for
every fixed~$K$---the estimation problem is genuinely
consistent, uniformly over~$\cF_K$.

\subsection{Adaptive minimax bounds}\label{ssec:adaptive}

For fixed~$K$, the upper and lower bounds in
Theorem~\ref{thm:main} differ by a factor of~$\log n$. This
apparent gap, however, conceals a finer structure: the upper bound is largest on \emph{near-Boolean}
functions~$f$ for which $I^{(2)}(f)\approx I(f)\approx K$, whereas the lower bound construction achieves
$I^{(2)}(f_\omega)\lesssim K^2/\sqrt{\log n}$
(see Section~\ref{sec:lower}).  Stratifying
$\cF_K$ by the value of $I^{(2)}(f)$ reveals that the actual gap
on each stratum is only~$\sqrt{\log n}$, and that the apparent
$\log n$ gap on~$\cF_K$ is an aggregation artifact.

\begin{theorem}[Adaptive minimax bounds]\label{thm:adaptive}
  For every $\varepsilon\in(0,1)$, there exist constants $c,C>0$
  (depending only on $\sigma$ and~$\varepsilon$) such that the
  following holds for all integers $n\geqslant 2$ and $d\geqslant 1$.
  For $K>0$ and $B\in(0,K]$, define the sub-class
  \begin{equation*}
    \cF_{K,B}\;:=\;\bigl\{f\in\cF_K\,:\,I^{(2)}(f)\leqslant B\bigr\}.
  \end{equation*}
  \begin{enumerate}[label=(\roman*)]
  \item \textbf{Upper bound.}\; For every $1\leqslant K\leqslant\sqrt{\log n}$,
  $1\leqslant B\leqslant K$, and $\log d\leqslant n^{1-\varepsilon}$,
  the estimator $\hat f_n$ of Theorem~\ref{thm:main}(i) satisfies
  \begin{equation*}
    \sup_{f\in\cF_{K,B}}R(\hat f_n,f)
    \;\leqslant\; C\,\frac{B}{\sqrt{\log n}}.
  \end{equation*}
  \item \textbf{Lower bound.}\; For every $K\leqslant c\sqrt{\log n}$,
    $B\in(0,K^2/\sqrt{\log n}]$, and $d\geqslant(1/c)\log n$,
    \begin{equation}\label{eq:adaptive-lower}
      \inf_{\tilde f}\;\sup_{f\in\cF_{K,B}}R(\tilde f,f)
      \;\geqslant\; c\,\frac{B}{\log n},
    \end{equation}
    where the infimum is over all estimators $\tilde f=\tilde f(\cD_n)$.
  \end{enumerate}
\end{theorem}

Part~(i) follows immediately from the refined bound
Theorem~\ref{thm:upper}(ii), since every $f\in\cF_{K,B}$
satisfies $I^{(2)}(f)\leqslant B$.  Part~(ii) is proved in
Section~\ref{sec:lower} by a variant of the lower bound
construction of Theorem~\ref{thm:main}(ii), where the scaling
parameter is chosen to saturate the constraint
$I^{(2)}(f)\leqslant B$ rather than $I(f)\leqslant K$.

The two bounds differ by a factor~$\sqrt{\log n}$, uniformly
over $B\in(0,K^2/\sqrt{\log n}]$.  Note that
$K^2/\sqrt{\log n}\leqslant K$ under the assumption
$K\leqslant c\sqrt{\log n}$ with $c\leqslant 1$, so the constraint $B\leqslant K^2/\sqrt{\log n}$ in~(ii) is
compatible with the definition $B\in(0,K]$ of the sub-class.

The original lower bound of Theorem~\ref{thm:main}(ii) is
recovered as the boundary case $B=K^2/\sqrt{\log n}$:
\[
  \inf_{\hat f}\sup_{f\in\cF_K}R(\hat f,f)
  \;\geqslant\;
  \inf_{\hat f}\sup_{f\in\cF_{K,\,K^2/\sqrt{\log n}}}R(\hat f,f)
  \;\geqslant\;
  c\,\frac{K^2}{(\log n)^{3/2}}.
\]
Closing the
remaining $\sqrt{\log n}$ gap on~$\cF_{K,B}$ is an open problem
discussed in Section~\ref{sec:discussion}.

\subsection{Notation and organization}\label{ssec:org}

Section~\ref{sec:sparsity} discusses the three roles of the
$L^1$-influence, its analogy with soft sparsity in linear
regression, and related work.
Section~\ref{sec:prelim} develops the Fourier analysis on the
hypercube and establishes the spectral concentration result for
monotone functions under an influence budget.
Section~\ref{sec:upper} constructs the estimator and proves the
upper bounds (Theorem~\ref{thm:main}(i) and
Theorem~\ref{thm:adaptive}(i)).
Section~\ref{sec:lower} constructs the lower bound families and
proves the lower bounds (Theorem~\ref{thm:main}(ii) and
Theorem~\ref{thm:adaptive}(ii)).
Section~\ref{sec:discussion} discusses the gap between the
bounds, extensions, and open problems.

\section{The $L^1$-influence as soft sparsity}\label{sec:sparsity}

The estimator achieving Theorem~\ref{thm:main}(i) exploits three
distinct properties of the $L^1$-influence, which together make
$I(\cdot)$ the right complexity measure for our problem.

\smallskip
\emph{Statistical tractability.}\;
By identity~\eqref{eq:inf-means}, each $I_i(f)$ is a simple
difference of conditional means, directly estimable from data at
parametric rate~$n^{-1/2}$, uniformly in~$i$.  This stands in
contrast with the $L^2$-influence $I_i^{(2)}(f)=\E[(\Delta_i
f(X))^2]$, which involves $f$ at two points simultaneously and
admits no such plug-in estimator.

\smallskip
\emph{Analytical bridge.}\;
Every $f:\Hd\to\R$ admits an expansion $f=\sum_{S\subseteq[d]}
\hat f(S)\,\chi_S$ in the orthonormal basis $\{\chi_S\}$ of
$L^2(\Hd,\mu)$ (the Fourier basis on the hypercube, recalled in
Section~\ref{ssec:fourier}).  For any $f\in\cF_K$, the Fourier
coefficients concentrate on low-degree subsets of the influential
coordinates of~$f$ (Proposition~\ref{prop:spectral}), in the spirit 
of Friedgut's junta theorem \citep{friedgut1998}.  The bridge
$I^{(2)}(f)\leqslant I(f)\leqslant K$ (see~\eqref{eq:L1-L2})
allows the $L^2$-tools of Boolean analysis to apply to the
$L^1$-constrained class~$\cF_K$, and the summation over
non-influential coordinates yields a factor~$K$ rather than~$d$
in the spectral bound, giving dimension-free control.

\smallskip
\emph{Budget control and estimation.}\;
The constraint $\sum_i I_i(f)\leqslant K$, combined with a
threshold on the estimated influences, limits the number of
selected coordinates to at most~$O(K/\delta)$ for a
threshold~$\delta$.  This keeps manageable the set of Fourier
coefficients to be estimated.  Our estimator is then simply
constructed by estimating each coefficient in this set by its
empirical counterpart, and truncating to~$[0,1]$.  The
bias-variance trade-off in the threshold~$\delta$ yields the
rate~$K/\sqrt{\log n}$.  Under an $L^2$-influence constraint,
the analogous bound on the number of selected coordinates would
be quadratically worse, leading to a much larger estimation set.

\medskip
The triple role of the $L^1$-influence---statistical
tractability, analytical bridge, and budget control---parallels
the role of the $\ell_1$ norm in high-dimensional linear
regression, where the constraint $\|\beta\|_1\leqslant K$
simultaneously enables support estimation, provides approximation
guarantees, and controls the complexity of the estimator.  The
analogy is summarized in the table below.

\begin{center}
\renewcommand{\arraystretch}{1.3}
\begin{tabular}{@{}lll@{}}
  & \textbf{Linear model} & \textbf{Monotone model on~$\Hd$}
  \\[2pt]
  Exact sparsity & $\|\beta\|_0\leqslant s$ & $f$ is an $s$-junta \\
  Soft sparsity & $\|\beta\|_1\leqslant K$ & $I(f)\leqslant K$ \\
  Key implication &
    $\|\beta\|_1\leqslant K\Rightarrow\beta\approx\text{sparse}$ &
    $I(f)\leqslant K\Rightarrow f\approx\text{junta}$ \\
  Structural result & $\ell_1$-ball geometry &
    Spectral concentration \\
  & & (Friedgut-type) \\
  Complexity measure & estimable at rate~$n^{-1/2}$ &
    estimable at rate~$n^{-1/2}$ \\
  Minimax rate & $K\sqrt{(\log d)/n}$ &
    $K/\sqrt{\log n}$ (this paper)
\end{tabular}
\end{center}

\noindent
In both settings, the soft constraint does not impose exact
low-dimensional structure but guarantees approximate
low-dimensionality, which suffices for estimation.  The analogy
extends to the proof strategy: influence estimation parallels
support recovery in the Lasso, and Fourier thresholding
parallels soft thresholding of regression coefficients.

\medskip\noindent\textbf{Related work.}\;
\leavevmode\par
\vspace{0.3em}

\emph{Isotonic regression}.
Estimation of monotone functions is a classical topic in
nonparametric statistics.  In dimension~$1$,
\citet{brunk1955} introduced the isotonic least squares estimator,
with $L^2$ risk of order~$n^{-2/3}$ \citep{vaneeden1958,brunk1970};
see \citet{groeneboom2014} for a modern treatment.  For
multivariate isotonic regression on~$[0,1]^d$,
\citet{chatterjee2015} established risk bounds for the least squares
estimator over partially ordered domains, and \citet{han2019} determined the minimax
rate: $n^{-1/d}$ up to logarithmic factors, which degrades
rapidly with~$d$.  On the Boolean hypercube, the influence constraint provides an
alternative mechanism for controlling the complexity of the
estimation problem, even when $d\gg\log n$.

\smallskip
\emph{Sparse nonparametric regression}.
Exploiting low-dimensional structure in high-dimen\-sional
nonparametric problems has been extensively studied under exact
sparsity, where the target depends on $s\ll d$ unknown coordinates
\citep{yang1999,kpotufe2011,comminges2012}.  Our work differs in
that $I(f)\leqslant K$ does not require~$f$ to depend on exactly~$s$
variables---it merely implies proximity to a junta. 
This is a strictly weaker assumption: Tribes, for instance,
depends on all~$d$ coordinates yet satisfies
$I=\Theta(\log d)$, and Majority satisfies
$I=\Theta(\sqrt{d})$.

\smallskip
\emph{Monotone classification}.
Learning monotone classifiers from binary data has received
attention in applied machine learning, including monotone decision
trees \citep{ben-david1995,potharst1999,feelders2010} and ensemble
methods \citep{gonzalez2015}.  These works focus on algorithms and
empirical performance rather than minimax rates.

\smallskip
\emph{Analysis of Boolean functions.}
The influence of a variable was introduced by \citet{kahn1988},
who proved that every Boolean function on~$\Hd$ has a
coordinate with influence $\Omega(\log d/d)$ (the KKL
inequality); for Boolean functions, the $L^1$- and
$L^2$-influences coincide.  \citet{friedgut1998} showed that
bounded total $L^2$-influence implies proximity to a junta;
extensions to monotone functions on general product spaces
were obtained by \citet{friedgut2004}, and to general
(not necessarily monotone) functions by \citet{hatami2009},
who showed that small total influence implies proximity to
a decision tree even without monotonicity.  The connection
between influences and sharp thresholds is surveyed
in~\citet{garban2014}.  We use spectral concentration
arguments originating in Friedgut's work, adapted to our
statistical setting (monotone real-valued functions under an
$L^1$-influence budget), where the distinction between $L^1$-
and $L^2$-influences is essential.

\smallskip
\emph{Learning juntas}.
PAC learnability of monotone functions and juntas has been studied
by \citet{bshouty1996}, \citet{mossel2003}, and
\citet{odonnell2007}, with a focus on computational complexity.
Our results are information-theoretic and establish minimax rates
without computational constraints.

\section{Fourier analysis and spectral concentration}\label{sec:prelim}

This section develops the analytical tools used in the proofs.
Section~\ref{ssec:fourier} recalls the Fourier expansion on the
Boolean hypercube, following \citet{odonnell2014} (transposed from
$\{-1,1\}^d$ to $\{0,1\}^d$ via $x_i\leftrightarrow 2x_i-1$).
Section~\ref{ssec:spectral} establishes a spectral concentration
result for monotone functions under an influence budget.

\subsection{Fourier expansion on the hypercube}\label{ssec:fourier}

For each $S\subseteq[d]$, define the character
$\chi_S:\Hd\to\{-1,+1\}$ by
\begin{equation*}
  \chi_S(x) := \prod_{i\in S}(2x_i-1),
  \qquad \chi_\emptyset\equiv 1.
\end{equation*}
The system $\{\chi_S\}_{S\subseteq[d]}$ forms an orthonormal basis
of $L^2(\Hd,\mu)$, and every $f:\Hd\to\R$ admits the Fourier
expansion
\begin{equation*}
  f(x) = \sum_{S\subseteq[d]}\hat f(S)\,\chi_S(x),
  \qquad
  \hat f(S) := \E_\mu[f(X)\,\chi_S(X)].
\end{equation*}
Parseval's identity gives
$\|f\|_2^2=\sum_{S\subseteq[d]}\hat f(S)^2$ and
$\Var(f)=\sum_{S\neq\emptyset}\hat f(S)^2$.

\medskip\noindent\textbf{Fourier representation of influences.}\;
For any $f:\Hd\to\R$ and $i\in[d]$, the discrete derivative
$\Delta_i f(x)=f(x^{i\to 1})-f(x^{i\to 0})$ satisfies
\begin{equation}\label{eq:deriv-fourier}
\Delta_i f(x) = 2\sum_{S\ni i}\hat f(S)\,\chi_{S\setminus\{i\}}(x).
\end{equation}
Indeed, for $i\notin S$ the character $\chi_S$ does not depend on
$x_i$, so $\chi_S(x^{i\to 1})=\chi_S(x^{i\to 0})$.  For $i\in S$,
writing $\chi_S=\chi_{\{i\}}\,\chi_{S\setminus\{i\}}$ and using
$\chi_{\{i\}}(x^{i\to 1})-\chi_{\{i\}}(x^{i\to 0})=1-(-1)=2$ gives
the result.

Squaring~\eqref{eq:deriv-fourier} and taking expectations (using
the orthonormality of the characters), we obtain
\begin{equation}\label{eq:inf2-fourier}
  I_i^{(2)}(f) = 4\sum_{S\ni i}\hat f(S)^2.
\end{equation}
For monotone~$f$, taking expectations
of~\eqref{eq:deriv-fourier} directly and noting that only the
term $S=\{i\}$ survives (since $\E_\mu[\chi_{S\setminus\{i\}}]=0$ for $S\neq\{i\}$), we get
\begin{equation*}
  I_i(f) = \E[\Delta_i f(X)] = 2\hat f(\{i\})\geqslant 0,
\end{equation*}
so every first-order Fourier coefficient satisfies
$\hat f(\{i\})\geqslant 0$. Summing~\eqref{eq:inf2-fourier} over~$i$, 
we also note that 
\begin{equation}\label{eq:totinf-fourier}
  I^{(2)}(f) = 4\sum_{S\neq\emptyset}|S|\,\hat f(S)^2.
\end{equation}

Finally, we recall the noise operator and the hypercontractive
inequality, which are central to the proof of spectral
concentration in Section~\ref{ssec:spectral}.  For $\rho\in[0,1]$,
the noise operator $T_\rho$ acts on $f:\Hd\to\R$ by
$T_\rho f(x)=\sum_{S\subseteq[d]}\rho^{|S|}\hat f(S)\,\chi_S(x)$
(it damps each Fourier coefficient by the factor
$\rho^{|S|}$).

\begin{theorem}[\citealt{bonami1970,beckner1975}]
  \label{thm:bonami}
  For $1\leqslant p\leqslant q$ and
  $0\leqslant\rho\leqslant\sqrt{\frac{p-1}{q-1}}$, every
  $f:\Hd\to\R$ satisfies
  $\|T_\rho f\|_{q}\leqslant\|f\|_{p}$,
  where $\|g\|_{r}:=(\E_\mu[|g(X)|^r])^{1/r}$.
\end{theorem}

\subsection{Spectral concentration under an influence budget}
\label{ssec:spectral}

The following result is the main analytical tool behind the upper
bound.  It shows that the Fourier spectrum of a monotone function
with bounded $L^1$-influence concentrates on low-degree subsets of
the influential coordinates, with explicit dimension-free bounds.
The proof adapts the strategy of \citet{friedgut1998}, who
established that Boolean functions with bounded $L^2$-influence
are well approximated by juntas, to our setting of monotone real-valued
functions under an $L^1$-influence constraint.  Two features are
specific to our setting.  First, the bound is stated in terms of the
$L^2$-influence $I^{(2)}(f)$ in the high-degree component
(Part~1 below), which yields the $K$-dependent uniform bound
through the bridge $I^{(2)}(f)\leqslant I(f)\leqslant K$ but is
strictly stronger in the \emph{fluid} regime $I^{(2)}(f)\ll K$.
Second, in the low-degree component (Part~2), the ambient
dimension~$d$ of the classical statement is replaced by the
influence budget~$K$, which is essential for dimension-free
estimation.

\begin{proposition}[Spectral concentration]\label{prop:spectral}
  Let $f:\Hd\to[0,1]$ be monotone with $I(f)\leqslant K$.   
  For any integer $d_0\geqslant 1$ and any $\delta>0$, define
  \begin{equation*}
    J := \bigl\{i\in[d] : I_i(f)\geqslant\delta\bigr\}
    \qquad\text{and}\qquad
    \cS := \bigl\{S\subseteq J : |S|\leqslant d_0\bigr\}.
  \end{equation*}
  Then:
  \smallskip
  \begin{enumerate}[label=(\roman*),nosep]
  \item $|J|\leqslant K/\delta$.
  \item $\displaystyle
    \sum_{S\notin\cS}\hat f(S)^2
    \;\leqslant\; \frac{I^{(2)}(f)}{4\,d_0}
    \;+\; \frac{K\cdot 3^{d_0}\cdot\delta^{1/2}}{12}$.
  \end{enumerate}
\end{proposition}

\begin{proof}
  Part~(i) is immediate: $\sum_{i\in J}I_i(f)\leqslant I(f)\leqslant K$
  and each term is $\geqslant\delta$, so $|J|\leqslant K/\delta$.

  For Part~(ii), we split the Fourier weight outside $\cS$ into
  high-degree terms and low-degree terms not supported on~$J$.

  \medskip\noindent
  \textbf{Part~1: High-degree terms.}\;
  By~\eqref{eq:totinf-fourier},
  \begin{equation}\label{eq:high-deg}
    \sum_{|S|>d_0}\hat f(S)^2
    \;\leqslant\;\frac{1}{d_0}\sum_{|S|>d_0}|S|\,\hat f(S)^2
    \;\leqslant\;\frac{I^{(2)}(f)}{4\,d_0}.
  \end{equation}

  \medskip\noindent
  \textbf{Part~2: Low-degree terms outside $J$.}\;
  Define
  $W:=\sum_{\substack{|S|\leqslant d_0,\;S\not\subseteq J}}\hat f(S)^2$.
  Every such $S$ contains some $i\notin J$, so
  $W\leqslant\sum_{i\notin J}W_i$ where
  $W_i:=\sum_{\substack{S\ni i,\;|S|\leqslant d_0}}\hat f(S)^2$.

 Fix $i\notin J$.  Since $\Delta_i f(x)=f(x^{i\to 1})-f(x^{i\to 0})$
  does not depend on~$x_i$, we may define
  $\varphi_i:\{0,1\}^{d-1}\to[0,\tfrac{1}{2}]$ by
  $\varphi_i(x_{-i}):=\Delta_i f(x)/2$, where
  $x_{-i}\in\{0,1\}^{d-1}$ denotes~$x$ with the $i$-th coordinate
  removed. Identity~\eqref{eq:deriv-fourier} then reads
$\varphi_i(x_{-i}) = \sum_{S\ni i}\hat f(S)\,
\chi_{S\setminus\{i\}}(x_{-i})$,
which is the Fourier expansion of~$\varphi_i$ on
$\{0,1\}^{d-1}$, with coefficients
$\widehat{\varphi_i}(S')=\hat f(S'\cup\{i\})$ for
$S'\subseteq[d]\setminus\{i\}$.
Therefore,
\begin{equation*}
  W_i = \sum_{|S'|\leqslant d_0-1}\widehat{\varphi_i}(S')^2.
\end{equation*}
  
We apply Theorem~\ref{thm:bonami} with $p=4/3$, $q=2$,
  $\rho=1/\sqrt{3}$ (since $(p-1)/(q-1)=1/3=\rho^2$).  This gives
  $\|T_{1/\sqrt{3}}\varphi_i\|_2\leqslant\|\varphi_i\|_{4/3}$.  
   Since
  \[
    \|T_{1/\sqrt{3}}\varphi_i\|_2^2
    = \sum_{S'\subseteq[d]\setminus\{i\}}3^{-|S'|}\widehat{\varphi_i}(S')^2
    \geqslant 3^{-(d_0-1)}\sum_{|S'|\leqslant d_0-1}\widehat{\varphi_i}(S')^2
    = 3^{-(d_0-1)}W_i,
  \]
 we obtain
  \begin{equation}\label{eq:Wi-hyper}
    W_i \;\leqslant\; 3^{d_0-1}\,\|\varphi_i\|_{4/3}^2.
  \end{equation}

  We now bound $\|\varphi_i\|_{4/3}^2$. Since $\Delta_i f\in[0,1]$ (by monotonicity and $f\in[0,1]$),
  we have $(\Delta_i f)^p\leqslant\Delta_i f$ pointwise for any
  $p\geqslant 1$, and therefore
  \[
\|\varphi_i\|_{4/3}^{4/3}
= \E\bigl[|\varphi_i(X_{-i})|^{4/3}\bigr]
= 2^{-4/3}\,\E\bigl[(\Delta_i f(X))^{4/3}\bigr]
\leqslant 2^{-4/3}\,I_i(f).
  \]
  Raising to the power $3/2$, we obtain
  \begin{equation*}
    \|\varphi_i\|_{4/3}^2
    \leqslant \bigl(2^{-4/3}\,I_i(f)\bigr)^{3/2}
    = \frac{I_i(f)^{3/2}}{4}.
  \end{equation*}
  Substituting into~\eqref{eq:Wi-hyper}, we are led to
  \begin{equation*}
    W_i \;\leqslant\; \frac{3^{d_0-1}}{4}\,I_i(f)^{3/2}.
  \end{equation*}

  For $i\notin J$: $I_i(f)<\delta$ by definition of~$J$.  Hence
  $I_i(f)^{3/2}=I_i(f)\cdot I_i(f)^{1/2}<I_i(f)\cdot\delta^{1/2}$.
  Summing over $i\notin J$, we conclude 
  \begin{equation}\label{eq:W-bound}
    W
    \;\leqslant\;\frac{3^{d_0-1}\,\delta^{1/2}}{4}
    \sum_{i\notin J}I_i(f)
    \;\leqslant\;\frac{3^{d_0-1}\,\delta^{1/2}}{4}\,K
    \;=\;\frac{K\cdot 3^{d_0}\cdot\delta^{1/2}}{12}.
  \end{equation}
  Combining~\eqref{eq:high-deg} and~\eqref{eq:W-bound} gives
  Part~(ii).
\end{proof}

\medskip\noindent\textbf{Comparison with the classical statement.}\;
  
The junta theorem of \citet{friedgut1998}, stated for Boolean
functions $f:\Hd\to\{0,1\}$ with bounded total $L^2$-influence
$I^{(2)}(f)\leqslant K$, asserts that $f$ is $\varepsilon$-close
in $L^2$ to a function depending on at most $2^{O(K/\varepsilon)}$
coordinates.  Proposition~\ref{prop:spectral} extends this to
monotone real-valued functions under an $L^1$-influence budget,
via the bridge $I^{(2)}(f)\leqslant I(f)\leqslant K$.

Specifically, fix $\varepsilon\in(0,1)$ and set
$d_0:=\lceil K/(2\varepsilon)\rceil$ and
$\delta:=e^{-\gamma d_0}$ with $\gamma>2\log 3$.  Define
$J:=\{i:I_i(f)\geqslant\delta\}$ and
$g:=\sum_{S\in\cS}\hat f(S)\,\chi_S$, where
$\cS:=\{S\subseteq J:|S|\leqslant d_0\}$.  By
Proposition~\ref{prop:spectral}(ii) and $I^{(2)}(f)\leqslant K$,
\[
  \|f-g\|_2^2
  \;\leqslant\;\frac{K}{4d_0}+\frac{K}{12}\,(3e^{-\gamma/2})^{d_0}
  \;\leqslant\;\frac{\varepsilon}{2}+\frac{\varepsilon}{2}
  \;=\;\varepsilon,
\]
provided $K/\varepsilon$ is larger than a constant depending
only on~$\gamma$ (so that the second term is also at
most~$\varepsilon/2$).  Since $g$ depends only on the
coordinates in~$J$, and Proposition~\ref{prop:spectral}(i)
gives $|J|\leqslant K/\delta=Ke^{\gamma d_0}$, the function~$g$
is a junta with at most
\[
  |J|\leqslant Ke^{\gamma\lceil K/(2\varepsilon)\rceil}
  \;=\;2^{O(K/\varepsilon)}
\]
coordinates, where the last equality uses
$\gamma d_0\leqslant\gamma(K/(2\varepsilon)+1)=O(K/\varepsilon)$
and $\log K\leqslant K/(2\varepsilon)$ for $K\geqslant 1$ and
$\varepsilon\leqslant 1/2$.  This recovers the classical junta
theorem with explicit, dimension-free bounds.

Our result provides two additional pieces of information
exploited in Section~\ref{sec:upper}.  First, the approximation
is through the Fourier projection onto~$\cS$, whose cardinality
is much smaller than that of the full junta.  Indeed, by
Proposition~\ref{prop:spectral}(i), $|J|\leqslant Ke^{\gamma d_0}$,
and the standard binomial estimate
$|\cS|\leqslant\sum_{j=0}^{d_0}\binom{|J|}{j}
\leqslant(e|J|/d_0)^{d_0}$ gives
\[
  \log|\cS|
  \;\leqslant\; d_0\log\frac{e|J|}{d_0}
  \;\leqslant\; d_0\bigl(\gamma d_0+\log(eK/d_0)\bigr)
  \;=\; \gamma d_0^2 + d_0\log(eK/d_0),
\]
which grows at most quadratically in~$d_0$, whereas the full
junta $\{0,1\}^J$ has $2^{|J|}=2^{2^{O(K/\varepsilon)}}$
elements---a double exponential in $K/\varepsilon$.  This
gap is the key reason the Fourier estimator achieves a better
rate than a junta-based estimator; the precise variance
calculation exploiting this bound is carried out in
Section~\ref{sec:upper}. Second, the bound in
Proposition~\ref{prop:spectral}(ii) involves the budget~$K$
rather than the ambient dimension~$d$, which is essential for
dimension-free estimation: in the original argument of
\citet{friedgut1998}, the analogous bound
carries a prefactor~$d$ instead of~$K$.

Finally, we note that the exponential dependence
$|J|=2^{O(K/\varepsilon)}$ cannot in general be improved.
The Tribes function (Example~\ref{ex:examples}(c)) satisfies
$I(\mathrm{Tribes})=\Theta(\log d)$, so $K=\Theta(\log d)$
and $d=2^{\Theta(K)}$.  Yet any $g$ with
$\|\mathrm{Tribes}-g\|_2^2\leqslant 1/10$ must depend on at
least $\Omega(d/\log d)$ coordinates
\citep[Section~4.2]{odonnell2014}, which is $2^{\Theta(K)}/\Theta(K)$---super-polynomial in~$K$, confirming that the exponential
dependence on~$K$ in the junta size is unavoidable.

\section{Proof of the upper bound}\label{sec:upper}

This section constructs the estimator and proves
Theorem~\ref{thm:main}(i) and Theorem~\ref{thm:adaptive}(i). 

The procedure combines influence
estimation with Fourier coefficient estimation on the spectral
concentration set identified in Proposition~\ref{prop:spectral}.
A sample-splitting step ensures independence between variable
selection and coefficient estimation.

\subsection{The estimator}\label{ssec:estimator}

The estimator proceeds as follows.

\medskip
\noindent\rule{\linewidth}{0.4pt}
\begin{procedure}[Fourier thresholding estimator]\label{proc:fourier}
  \leavevmode\par
\vspace{0.3em}
\noindent\emph{Input}: data $\cD_n=\{(X_j,Y_j)\}_{j=1}^n$;
 universal constants $\gamma>2\log 3$ and $c_0>0$.

\smallskip
\noindent\emph{Step~0 (Sample splitting).}\;
  Split $\cD_n$ into $\cD_1=\{(X_j,Y_j)\}_{j=1}^{n_1}$ and
  $\cD_2=\{(X_j,Y_j)\}_{j=n_1+1}^{n}$, where
  $n_1=\lfloor n/2\rfloor$ and $n_2=n-n_1$.

\smallskip
\noindent\emph{Step~1 (Variable selection).}\;
  Set
  \begin{equation*}
    d_0 := \max\!\left(\left\lceil\sqrt{\frac{(\log n
- c_0\sqrt{\log n})_+}{\gamma}}\;\right\rceil,\, 1\right)
    \qquad\text{and}\qquad
    \delta:=e^{-\gamma d_0}.
  \end{equation*}
  Using $\cD_1$, compute for each $i\in[d]$:
  \begin{equation*}
    \hat I_i := \bar Y_{i,1}-\bar Y_{i,0},
    \qquad
    \bar Y_{i,b}
    := \frac{\sum_{j\leqslant n_1}Y_j\,\one_{\{X_{j,i}=b\}}}
    {\sum_{j\leqslant n_1}\one_{\{X_{j,i}=b\}}}.
  \end{equation*}
  Select coordinates
  $\hat J := \{i\in[d]:\hat I_i\geqslant\delta/2\}$.

\smallskip
\noindent\emph{Step~2 (Fourier estimation).}\;
  Using $\cD_2$, compute the empirical Fourier coefficients
  \[
    \tilde f(S)
    := \frac{1}{n_2}\sum_{j=n_1+1}^{n}Y_j\,\chi_S(X_j)
  \]
  for every $S$ in the estimated spectral set
  $\hat\cS:=\{S\subseteq\hat J:|S|\leqslant d_0\}$.

\smallskip
\noindent\emph{Step~3 (Reconstruction and truncation).}\; Output
  \begin{equation*}
    \hat f_n(x) := \max\!\Big(0,\;\min\!\Big(
    \sum_{S\in\hat\cS}\tilde f(S)\,\chi_S(x),\;1\Big)\Big).
  \end{equation*}
\end{procedure}
\nopagebreak
\noindent\rule{\linewidth}{0.4pt}
\medskip

We emphasize that $d_0$ depends only on~$n$, $\gamma$, and~$c_0$
(universal constants), not on~$K$ or~$d$.  The estimator $\hat f_n$
therefore does not require knowledge of~$K$ and achieves the bound
of Theorem~\ref{thm:main} simultaneously for all $K\leqslant d$
and all $f\in\cF_K$.

\begin{proposition}\label{prop:inf-conc}
  There exists $C_\sigma>0$ (depending only on $\sigma$) such that
  for every $t>0$, 
  \begin{equation*}
    \mathbb P\!\left(\max_{i\in[d]}|\hat I_i-I_i(f)|>t\right)
    \;\leqslant\; 2d\exp\!\left(-\frac{n_1\,t^2}{C_\sigma}\right)
    + 2d\,e^{-n_1/8}.
  \end{equation*}
\end{proposition}

\begin{proof}
  Fix $i\in[d]$ and let
  $N_{i,b}:=\sum_{j\leqslant n_1}\one_{\{X_{j,i}=b\}}$.
  Since $N_{i,1}\sim\mathrm{Bin}(n_1,1/2)$, Hoeffding's inequality
  \citep[Theorem~2.2.6]{vershynin2018} gives
  \[
    \mathbb P(N_{i,1}\leqslant n_1/4)
    \;\leqslant\; e^{-n_1/8},
  \]
  and the same bound holds for $N_{i,0}=n_1-N_{i,1}$.  Define the
  event $\cE_i:=\{N_{i,0}\geqslant n_1/4,\;N_{i,1}\geqslant n_1/4\}$,
  which satisfies $\mathbb P(\cE_i)\geqslant 1-2e^{-n_1/8}$.  On
  $\cE_i$, all sample means $\bar Y_{i,b}$ are well defined since
  $N_{i,b}\geqslant 1$.

  Conditionally on $N_{i,b}=n_b$, the observations
  $(X_j,Y_j)_{j:X_{j,i}=b}$ are $n_b$ i.i.d.\ copies of
  $(X,Y)$ drawn from the law of $(X,f(X)+\varepsilon)$ conditioned
  on $X_i=b$.  Each centered variable
  $Y_j-\E[f(X)\mid X_i=b]=(f(X_j)-\E[f(X)\mid X_i=b])+\varepsilon_j$
  is a sum of two independent terms: $\varepsilon_j$, which is
  sub-Gaussian with parameter~$\sigma^2$ by assumption, and
  $f(X_j)-\E[f(X)\mid X_i=b]$, which takes values in an interval
  of length at most~$1$ (since $f\in[0,1]$) and is therefore
  sub-Gaussian with parameter~$1/4$ by Hoeffding's lemma.  Hence
  $Y_j-\E[f(X)\mid X_i=b]$ is sub-Gaussian with
  parameter~$\sigma^2+1/4$, and $\bar Y_{i,b}-\E[f(X)\mid X_i=b]$
  is sub-Gaussian with parameter
  $(\sigma^2+1/4)/n_b\leqslant 4(\sigma^2+1/4)/n_1$ on~$\cE_i$.

Since $\hat I_i-I_i=(\bar Y_{i,1}-\E[f(X)\mid X_i{=}1])-
(\bar Y_{i,0}-\E[f(X)\mid X_i{=}0])$ is a difference of two
independent sub-Gaussian variables, it is sub-Gaussian with
parameter at most $8(\sigma^2+1/4)/n_1$ on $\cE_i$.
The standard tail bound then gives
\[
  \mathbb P\bigl(|\hat I_i-I_i|>t\;\big|\;\cE_i\bigr)
  \;\leqslant\; 2\exp\!\left(-\frac{n_1\,t^2}{C_\sigma}\right)
\]
with $C_\sigma:=16(\sigma^2+1)$, using $\sigma^2+1/4\leqslant\sigma^2+1$.  Since
  $\mathbb P(|\hat I_i-I_i|>t)\leqslant
  \mathbb P(|\hat I_i-I_i|>t\mid\cE_i)+\mathbb P(\cE_i^c)$,
  a union bound over $i\in[d]$ gives
  \begin{equation*}
    \mathbb P\!\left(\max_{i\in[d]}|\hat I_i-I_i(f)|>t\right)
    \;\leqslant\; 2d\exp\!\left(-\frac{n_1\,t^2}{C_\sigma}\right)
    + 2d\,e^{-n_1/8}.
  \end{equation*}
\end{proof}

Define the \emph{good event}
\begin{equation*}
  \Omega_n := \left\{\max_{i\in[d]}|\hat I_i-I_i(f)|
  \leqslant\frac{\delta}{4}\right\}.
\end{equation*}
Setting $t=\delta/4$ in Proposition~\ref{prop:inf-conc}, we have
\begin{equation*}
  \mathbb P(\Omega_n^c)
  \;\leqslant\; 2d\exp\!\left(-\frac{n_1\,\delta^2}
  {16\,C_\sigma}\right) + 2d\,e^{-n_1/8}.
\end{equation*}

\begin{proposition}[Variable selection]\label{prop:selection}
  On the event $\Omega_n$, the following properties hold.
  \begin{enumerate}[label=(\roman*),nosep]
  \item The set $J=\{i\in[d]:I_i(f)\geqslant\delta\}$
    from Proposition~\ref{prop:spectral} satisfies
    $J\subseteq\hat J$.
  \item Every $i\in\hat J$ satisfies
    $I_i(f)\geqslant\delta/4$.
  \item $|\hat J|\leqslant 4K/\delta$.
  \end{enumerate}
\end{proposition}

\begin{proof}
  On $\Omega_n$, we have
  $|\hat I_i-I_i(f)|\leqslant\delta/4$ for all $i\in[d]$.

  For~(i), let $i\in J$.  Then $I_i(f)\geqslant\delta$, so
  $\hat I_i\geqslant I_i(f)-\delta/4\geqslant 3\delta/4>\delta/2$,
  hence $i\in\hat J$.

  For~(ii), let $i\in\hat J$.  Then $\hat I_i\geqslant\delta/2$,
  hence $I_i(f)\geqslant\hat I_i-\delta/4\geqslant\delta/4$.

  For~(iii), summing the bound from~(ii) over $i\in\hat J$ gives
  $\sum_{i\in\hat J}I_i(f)\geqslant|\hat J|\cdot\delta/4$.
  Since $\sum_{i=1}^d I_i(f)=I(f)\leqslant K$, we conclude that
  $|\hat J|\leqslant 4K/\delta$.
\end{proof}

\subsection{Risk analysis}\label{ssec:risk}

The following result is the technical core of the upper bound.
It implies Theorem~\ref{thm:main}(i) via Part~(i) and
Theorem~\ref{thm:adaptive}(i) via Part~(ii).

\begin{theorem}[Upper bound, detailed version]\label{thm:upper}
  For every $\varepsilon\in(0,1)$, there exist constants $c,C,C'>0$
  (depending only on $\sigma$ and~$\varepsilon$) such that the
  following holds for all integers $n\geqslant 2$, $d\geqslant 1$
  with $\log d\leqslant n^{1-\varepsilon}$, and every $K\leqslant d$.
  \begin{enumerate}[label=(\roman*)]
  \item \textbf{Uniform bound.}\;
    \begin{equation*}
      R(\hat f_n,f)
      \;\leqslant\; C\,\frac{K}{\sqrt{\log n}}
      \;+\; C'\,e^{-c\sqrt{\log n}}.
    \end{equation*}
 \item \textbf{Refined bound.}\;
    Under the additional assumption $K\leqslant\sqrt{\log n}$,
    \begin{equation*}
      R(\hat f_n,f)
      \;\leqslant\; C\,\frac{I^{(2)}(f)}{\sqrt{\log n}}
      \;+\; C'\,e^{-c\sqrt{\log n}}.
    \end{equation*}
  \end{enumerate}
\end{theorem}

The two bounds are consistent: Part~(i) follows from
Part~(ii) via $I^{(2)}(f)\leqslant I(f)\leqslant K$
(see~\eqref{eq:L1-L2}).  Part~(ii) is strictly stronger
in the \emph{fluid} regime $I^{(2)}(f)\ll K$, and will
be used to derive Theorem~\ref{thm:adaptive}(i) in
Section~\ref{ssec:proof-upper}.

\begin{proof}
Throughout, $C_1,C_2,\ldots$ denote positive constants depending
only on $\sigma$, $\gamma$, $\varepsilon$, and~$c_0$; their
values may change from line to line but never depend on $n$,
$d$, or~$K$.  All conditions of the form ``for $n$ large
enough'' involve thresholds depending only on these parameters
and can be absorbed by enlarging the constants $C$ and~$C'$
in the statement; in particular, we assume henceforth that
$n$ is large enough that $\log n\geqslant c_0\sqrt{\log n}$,
so that $d_0=\bigl\lceil\sqrt{(\log n-c_0\sqrt{\log n})/\gamma}
\bigr\rceil\geqslant 1$.  The constant~$c_0$ depends only
on~$\gamma$ and is determined at the end of Step~1.

For $K>\sqrt{\log n}$, Part~(i) holds since
$CK/\sqrt{\log n}\geqslant C\geqslant 1\geqslant R(\hat f_n,f)$
for any $C\geqslant 1$.  It therefore suffices to prove
Part~(ii) under the assumption
\begin{equation}\label{eq:K-small}
  K\;\leqslant\;\sqrt{\log n},
\end{equation}
from which Part~(i) follows via $I^{(2)}(f)\leqslant K$.

Since $\hat f_n$ is truncated to $[0,1]$ and $f\in[0,1]$, the
pointwise error satisfies
$|\hat f_n(x)-f(x)|\leqslant 1$ for all~$x$, hence
\begin{equation}\label{eq:trivial-bound}
  \|\hat f_n-f\|_2^2\;\leqslant\; 1 \qquad\text{always.}
\end{equation}
We decompose the risk as
\begin{equation}\label{eq:risk-decomp}
  R(\hat f_n,f)
  \;=\; \E\bigl[\|\hat f_n-f\|_2^2\,\one_{\Omega_n}\bigr]
  + \E\bigl[\|\hat f_n-f\|_2^2\,\one_{\Omega_n^c}\bigr]
  \;\leqslant\;
  \E\bigl[\|\hat f_n-f\|_2^2\,\one_{\Omega_n}\bigr]
  + \mathbb P(\Omega_n^c),
\end{equation}
where the inequality uses~\eqref{eq:trivial-bound}.

\medskip\noindent\textbf{Step~1: Risk on the good event
(bias\,$+$\,variance).}\;
On~$\Omega_n$, the truncation can only reduce the $L^2$~error
(projecting onto $[0,1]$ is a contraction when the target lies
in~$[0,1]$), so it suffices to bound the risk of the
untruncated estimator
$\tilde f_n=\sum_{S\in\hat\cS}\tilde f(S)\chi_S$.  By
Parseval's identity, since $\tilde f_n$ sets to zero all
Fourier coefficients outside~$\hat\cS$,
\[
  \|\tilde f_n-f\|_2^2
  \;=\; \underbrace{\sum_{S\in\hat\cS}
  \bigl(\tilde f(S)-\hat f(S)\bigr)^2}_{\text{variance}}
  \;+\; \underbrace{\sum_{S\notin\hat\cS}
  \hat f(S)^2}_{\text{bias}}.
\]

\smallskip\noindent\emph{Bias.}\;
On $\Omega_n$, Proposition~\ref{prop:selection}(i) gives
$J\subseteq\hat J$, hence
$\cS=\{S\subseteq J:|S|\leqslant d_0\}\subseteq\hat\cS$.
By Proposition~\ref{prop:spectral}(ii) with
$\delta=e^{-\gamma d_0}$,
\begin{equation*}
  \sum_{S\notin\hat\cS}\hat f(S)^2
  \;\leqslant\;\sum_{S\notin\cS}\hat f(S)^2
  \;\leqslant\;\frac{I^{(2)}(f)}{4\,d_0}
  +\frac{K}{12}\,\eta^{d_0},
\end{equation*}
where $\eta:=3e^{-\gamma/2}<1$ (since $\gamma>2\log 3$).
Since $d_0\geqslant C_1\sqrt{\log n}$, the first term
satisfies 
\[
  \frac{I^{(2)}(f)}{4d_0}
  \;\leqslant\; \frac{C_2\,I^{(2)}(f)}{\sqrt{\log n}}.
\]
The second term decays exponentially: since
$K\leqslant\sqrt{\log n}$ by~\eqref{eq:K-small},
$(K/12)\eta^{d_0}\leqslant C_3\,e^{-C_4\sqrt{\log n}}$.
Therefore,
\begin{equation}\label{eq:bias-final}
  \text{Bias}
  \;\leqslant\;C_2\,\frac{I^{(2)}(f)}{\sqrt{\log n}}
  +C_3\,e^{-C_4\sqrt{\log n}}
  \;\leqslant\;C_2\,\frac{K}{\sqrt{\log n}}
  +C_3\,e^{-C_4\sqrt{\log n}},
\end{equation}
using $I^{(2)}(f)\leqslant K$ in the last step.

\smallskip\noindent\emph{Variance.}\;
By sample splitting, $\hat\cS$ is determined by~$\cD_1$ and
the coefficients $\tilde f(S)$ are computed from the
independent sample~$\cD_2$.  Conditionally on~$\cD_1$, the
errors $\tilde f(S)-\hat f(S)$ are centered, so by linearity
of expectation,
\[
  \E\!\left[\sum_{S\in\hat\cS}(\tilde f(S)-\hat f(S))^2
    \;\middle|\;\cD_1\right]
  = \sum_{S\in\hat\cS}\Var(\tilde f(S))
  \leqslant\frac{2(\sigma^2+1)}{n}\,|\hat\cS|,
\]
since $\Var(\tilde f(S))=\Var(Y\chi_S(X))/n_2
\leqslant\E[Y^2]/n_2\leqslant 2(\sigma^2+1)/n$ (using
$n_2\geqslant n/2$ and $\E[Y^2]\leqslant 1+\sigma^2$).

It remains to bound $|\hat\cS|$ on~$\Omega_n$.  By
Proposition~\ref{prop:selection}(iii),
$|\hat J|\leqslant 4Ke^{\gamma d_0}$.  We distinguish two
cases.  If $d_0\leqslant|\hat J|$, the standard binomial
inequality $\sum_{j=0}^{k}\binom{N}{j}\leqslant(eN/k)^{k}$
(valid for $k\leqslant N$) with $k=d_0$ and $N=|\hat J|$ gives
\[
  \log|\hat\cS|
  \;\leqslant\;d_0\bigl(\gamma d_0+\log(4eK/d_0)\bigr)
  \;=\;\gamma d_0^2+d_0\log(4eK/d_0).
\]
If $d_0>|\hat J|$, then $\hat\cS$ is the power set
of~$\hat J$ and $\log|\hat\cS|=|\hat J|\log 2
\leqslant d_0\log 2\leqslant\gamma d_0^2$
(since $d_0\geqslant 1$ and $\gamma>\log 2$).
In both cases,
\begin{equation}\label{eq:log-Shat}
  \log|\hat\cS|
  \;\leqslant\;\gamma d_0^2
  +d_0\bigl(\log(4eK/d_0)\bigr)^+.
\end{equation}

Write $u:=\sqrt{(\log n-c_0\sqrt{\log n})/\gamma}$, so that
$d_0=\lceil u\rceil\leqslant u+1$.  Then
\begin{equation}\label{eq:gd02}
  \gamma d_0^2
  \;\leqslant\;\gamma(u+1)^2
  \;=\;\log n-c_0\sqrt{\log n}+2\gamma u+\gamma
  \;\leqslant\;\log n-(c_0-2\sqrt{\gamma})\sqrt{\log n}+\gamma,
\end{equation}
where we used $u\leqslant\sqrt{(\log n)/\gamma}$.
By~\eqref{eq:K-small}, $K/d_0\leqslant\sqrt{\log n}/u
\leqslant\sqrt{2\gamma}$ for $n$ large enough that
$u\geqslant\sqrt{(\log n)/(2\gamma)}$, so
$4eK/d_0\leqslant 4e\sqrt{2\gamma}=:C_5$, and
\begin{equation}\label{eq:subleading}
  d_0\bigl(\log(4eK/d_0)\bigr)^+
  \;\leqslant\;(\log C_5)\,d_0
  \;\leqslant\;C_6\sqrt{\log n}.
\end{equation}
Substituting~\eqref{eq:gd02} and~\eqref{eq:subleading}
into~\eqref{eq:log-Shat},
\begin{equation}\label{eq:Shat-over-n}
  \log\frac{|\hat\cS|}{n}
  \;\leqslant\;-(c_0-2\sqrt{\gamma}-C_6)\sqrt{\log n}+\gamma.
\end{equation}
We now choose $c_0:=2\sqrt{\gamma}+C_6+2$, a constant
depending only on~$\gamma$.  Then the right-hand side
of~\eqref{eq:Shat-over-n} is at most
$-2\sqrt{\log n}+\gamma\leqslant-\sqrt{\log n}$ for $n$
large enough that $\sqrt{\log n}\geqslant\gamma$, and the
variance satisfies
\begin{equation}\label{eq:var-final}
  \text{Variance}
  \;\leqslant\;\frac{2(\sigma^2+1)}{n}\,|\hat\cS|
  \;\leqslant\;C_7\,e^{-\sqrt{\log n}}.
\end{equation}
Combining~\eqref{eq:bias-final} and~\eqref{eq:var-final},
\begin{equation}\label{eq:good-risk}
  \E\bigl[\|\hat f_n-f\|_2^2\,\one_{\Omega_n}\bigr]
  \;\leqslant\;C_2\,\frac{I^{(2)}(f)}{\sqrt{\log n}}
  +(C_3+C_7)\,e^{-c'\sqrt{\log n}},
\end{equation}
where $c':=\min(C_4,1)$.

\medskip\noindent\textbf{Step~2: Probability of the bad
event.}\;
By Proposition~\ref{prop:inf-conc} with $t=\delta/4$,
\begin{equation*}
  \mathbb P(\Omega_n^c)
  \;\leqslant\;\underbrace{2d\exp\!\left(
  -\frac{n_1\delta^2}{16C_\sigma}\right)}_{=:\,T_1}
  \;+\;\underbrace{2d\,e^{-n_1/8}}_{=:\,T_2},
\end{equation*}
where $C_\sigma=16(\sigma^2+1)$, and we use
$n_1=\lfloor n/2\rfloor\geqslant n/4$ throughout.

\smallskip\noindent\emph{Bound on~$T_2$.}\;
\[
  \log T_2
  \;=\;\log 2+\log d-\frac{n_1}{8}
  \;\leqslant\; 1+n^{1-\varepsilon}-\frac{n}{32}.
\]
Since $n^{1-\varepsilon}=o(n)$, the right-hand side is at most
$-n/64$ for $n$ large enough, giving
$T_2\leqslant e^{-n/64}\leqslant e^{-\sqrt{\log n}}$.

\smallskip\noindent\emph{Bound on~$T_1$.}\;
Since $d_0\leqslant\sqrt{(\log n)/\gamma}+1$,
\[
  \delta^2=e^{-2\gamma d_0}
  \;\geqslant\;e^{-2\gamma}\cdot e^{-2\sqrt{\gamma\log n}}.
\]
Setting $\alpha(n):=C_8\exp(\log n-2\sqrt{\gamma\log n})$
with $C_8:=e^{-2\gamma}/(64C_\sigma)$,
\[
  \frac{n_1\delta^2}{16C_\sigma}\geqslant\alpha(n),
  \qquad
  \log T_1\leqslant 1+n^{1-\varepsilon}-\alpha(n).
\]
Since $n^{1-\varepsilon}/\alpha(n)=C_8^{-1}
\exp(-\varepsilon\log n+2\sqrt{\gamma\log n})\to 0$,
we have $n^{1-\varepsilon}\leqslant\alpha(n)/2$ for $n$
large enough, so $\log T_1\leqslant 1-\alpha(n)/2$.
Finally, $\log n-2\sqrt{\gamma\log n}\geqslant(\log n)/2$
for $n$ large enough, so
$\alpha(n)\geqslant C_8\sqrt{n}\geqslant 2(1+\sqrt{\log n})$,
giving $T_1\leqslant e^{-\sqrt{\log n}}$.

\smallskip
Combining,
\begin{equation}\label{eq:bad-final}
  \mathbb P(\Omega_n^c)
  \;\leqslant\;C_9\,e^{-\sqrt{\log n}}.
\end{equation}

\medskip\noindent\textbf{Step~3: Conclusion.}\;
Substituting~\eqref{eq:good-risk} and~\eqref{eq:bad-final}
into~\eqref{eq:risk-decomp},
\begin{equation*}
  R(\hat f_n,f)
  \;\leqslant\;C_2\,\frac{I^{(2)}(f)}{\sqrt{\log n}}
  +(C_3+C_7+C_9)\,e^{-c'\sqrt{\log n}},
\end{equation*}
which is Part~(ii) with $C=C_2$ and $C'=C_3+C_7+C_9$.
Part~(i) follows by $I^{(2)}(f)\leqslant K$
for $K\leqslant\sqrt{\log n}$, and by
$CK/\sqrt{\log n}\geqslant 1\geqslant R(\hat f_n,f)$
for $K>\sqrt{\log n}$.
\end{proof}

\subsection{Proofs of Theorem~\ref{thm:main}(i) and
Theorem~\ref{thm:adaptive}(i)}\label{ssec:proof-upper}

Both results follow from Theorem~\ref{thm:upper} by the
same argument: the exponential term $C'e^{-c\sqrt{\log n}}$
is absorbed into the leading term after enlarging~$C$,
using $K/\sqrt{\log n}\gg e^{-c\sqrt{\log n}}$ for
$K\geqslant 1$ (Theorem~\ref{thm:main}(i), via Part~(i))
and $B/\sqrt{\log n}\gg e^{-c\sqrt{\log n}}$ for
$B\geqslant 1$ (Theorem~\ref{thm:adaptive}(i), via Part~(ii),
since $I^{(2)}(f)\leqslant B$ for every $f\in\cF_{K,B}$).\qed

\section{Proof of the lower bound}\label{sec:lower}

This section proves Theorem~\ref{thm:main}(ii) and
Theorem~\ref{thm:adaptive}(ii).  The argument
constructs a large family of well-separated monotone functions in
$\cF_K$ and applies Fano's inequality.  The construction exploits
the middle layer of the Boolean hypercube and the
Varshamov--Gilbert bound.

We use the following standard form of Fano's inequality; see
\citet[Theorem~2.7]{tsybakov2009}.

\begin{theorem}[Fano]\label{thm:fano}
  Let \(M\ge 2\), let $\{P_0,P_1,\ldots,P_M\}$ be probability measures on a
  measurable space, and let $\theta_0,\ldots,\theta_M$ be elements
  of a pseudometric space $(T,d)$ satisfying
  $d(\theta_i,\theta_j)\geqslant 2\delta>0$ for all $i\ne j$.
  Then
  \begin{equation}\label{eq:fano}
    \inf_{\hat\theta}\;\max_{0\leqslant j\leqslant M}\;
    \mathbb P_j\bigl(d(\hat\theta,\theta_j)\geqslant\delta\bigr)
    \;\geqslant\; 1-\frac{\bar\KL+\log 2}{\log M},
  \end{equation}
  where
  $\bar\KL:=\frac{1}{M}\sum_{j=1}^M\KL(P_j\|P_0)$.
\end{theorem}

\subsection{The construction}\label{ssec:construction}

Throughout, all inequalities hold for $n$ large enough
  (depending only on $\sigma$); the constant~$c$
  in the statement of Theorem~\ref{thm:main}(ii) is adjusted
  accordingly.

Fix $K\geqslant 1$ and $d\geqslant s$, where
\begin{equation*}
  s := \lfloor 2\log_2 n\rfloor.
\end{equation*}
For $z\in\{0,1\}^s$, write $|z|:=\sum_{i=1}^s z_i$ for its
Hamming weight.  Let $m:=\lfloor s/2\rfloor$ and
$L_m:=\{z\in\{0,1\}^s:|z|=m\}$ the middle layer, with
$N:=|L_m|=\binom{s}{m}$.

\medskip\noindent\textbf{Step~1: Packing on the middle layer.}\;
By the Varshamov--Gilbert bound
\citep[Lemma~4.7]{massart2007}, there exists a subset
$\Omega\subseteq\{0,1\}^N$ with
\begin{equation}\label{eq:VG}
  \log|\Omega|\;\geqslant\;\frac{N}{8}
  \;=\;\frac{1}{8}\binom{s}{m},
  \qquad
  d_H(\omega,\omega')\;\geqslant\;\frac{N}{4}
  \quad\text{for all }\omega\ne\omega'\in\Omega,
\end{equation}
where $d_H$ denotes the Hamming distance on $\{0,1\}^N$.  We
index the coordinates of $\omega\in\{0,1\}^N$ by the elements
of~$L_m$, writing $\omega_a$ for $a\in L_m$.

\medskip\noindent\textbf{Step~2: From binary vectors to monotone functions.}\;
Fix a subset $S_0\subseteq[d]$ with $|S_0|=s$, and set
\begin{equation*}
  \beta := \frac{K}{A\sqrt{s}},
\end{equation*}
where $A>0$ is an absolute constant determined by
Stirling's approximation, chosen
at the end of the influence calculation below to ensure
$I(f_\omega)\leqslant K$.

For each $\omega\in\Omega$, define $f_\omega:\Hd\to[0,1]$ by
\begin{equation*}
  f_\omega(x) :=
  \begin{cases}
    0 & \text{if }|x_{S_0}|<m,\\[2pt]
    \beta\,\omega_{x_{S_0}} & \text{if }|x_{S_0}|=m,\\[2pt]
    \beta & \text{if }|x_{S_0}|>m,
  \end{cases}
\end{equation*}
where $x_{S_0}\in\{0,1\}^s$ denotes the projection of $x$ onto
the coordinates in~$S_0$;
see Figure~\ref{fig:construction} for an illustration.
\begin{figure}[ht]
\centering
\begin{tikzpicture}[
    layer/.style={minimum width=7cm, minimum height=0.45cm,
                  draw, rounded corners=2pt},
    braceleft/.style={decorate,
                      decoration={brace, amplitude=5pt}},
  ]

  \node at (0, 5.4) {\small
    Sub-hypercube $\{0,1\}^{S_0}$, organized by
    $|x_{S_0}|=\sum_{i\in S_0}x_i$};

  \node[layer, fill=white] (l0) at (0, 0) {};
  \node[layer, fill=white] (l2) at (0, 1.30) {};

  \node[layer, fill=gray!15, line width=1pt] (lm) at (0, 2.30) {};

  \foreach \x/\v in {%
    -2.8/1, -2.1/0, -1.4/1, -0.7/1, 0/0, 0.7/0, 1.4/1, 2.1/0, 2.8/0}
     {
    \ifnum\v=1
      \fill[black!70] (\x-0.28, 2.30-0.17)
                       rectangle (\x+0.28, 2.30+0.17);
    \else
      \fill[white] (\x-0.28, 2.30-0.17)
                    rectangle (\x+0.28, 2.30+0.17);
      \draw[gray]  (\x-0.28, 2.30-0.17)
                    rectangle (\x+0.28, 2.30+0.17);
    \fi
  }

  \node[layer, fill=black!25] (lm1) at (0, 3.30) {};
  \node[layer, fill=black!25] (lm3) at (0, 4.60) {};

  \node at (0, 0.75) {\small $\vdots$};
  \node at (0, 4.05) {\small $\vdots$};

  \node[left] at (-3.8, 0)    {\small $|x_{S_0}|=0$};
  \node[left] at (-3.8, 1.30) {\small $|x_{S_0}|=m{-}1$};
  \node[left] at (-3.8, 2.30) {\small $|x_{S_0}|=m$};
  \node[left] at (-3.8, 3.30) {\small $|x_{S_0}|=m{+}1$};
  \node[left] at (-3.8, 4.60) {\small $|x_{S_0}|=s$};

  \node[right] at (3.8, 0)    {\small $f_\omega=0$};
  \node[right] at (3.8, 1.30) {\small $f_\omega=0$};
  \node[right] at (3.8, 2.30) {\small $f_\omega=\beta\,\omega_{x_{S_0}}$};
  \node[right] at (3.8, 3.30) {\small $f_\omega=\beta$};
  \node[right] at (3.8, 4.60) {\small $f_\omega=\beta$};

  \draw[braceleft] (-3.65, -0.25) -- (-3.65, 1.55)
    node[midway, left=6pt] {\small $f_\omega\equiv 0$};

  \draw[braceleft] (-3.65, 3.05) -- (-3.65, 4.85)
    node[midway, left=6pt] {\small $f_\omega\equiv\beta$};

  \node[anchor=west] at (-2.0, -1.0) {%
    \small
    \tikz{\fill[black!70] (0,0) rectangle (0.35,0.25);}\;$\omega_a=1$
    \qquad
    \tikz{\draw[gray] (0,0) rectangle (0.35,0.25);}\;$\omega_a=0$
  };

\end{tikzpicture}
\caption{Structure of the functions $f_\omega$ in the lower-bound
  construction.  Only the $s$~coordinates in~$S_0$ matter;
  the remaining $d-s$ coordinates do not affect~$f_\omega$.
  The sub-hypercube $\{0,1\}^{S_0}$ is organized by Hamming
  weight $|x_{S_0}|$.  Below the middle layer ($|x_{S_0}|<m$),
  $f_\omega\equiv 0$; above ($|x_{S_0}|>m$),
  $f_\omega\equiv\beta$.    On the middle layer~$L_m$, $f_\omega(a)=\beta$ if
  $\omega_a=1$ and $f_\omega(a)=0$ if $\omega_a=0$.
 Different choices of $\omega\in\Omega$ (the
  Varshamov--Gilbert packing) yield functions that differ
  only on~$L_m$ and are well separated in~$L^2$.}
\label{fig:construction}
\end{figure}
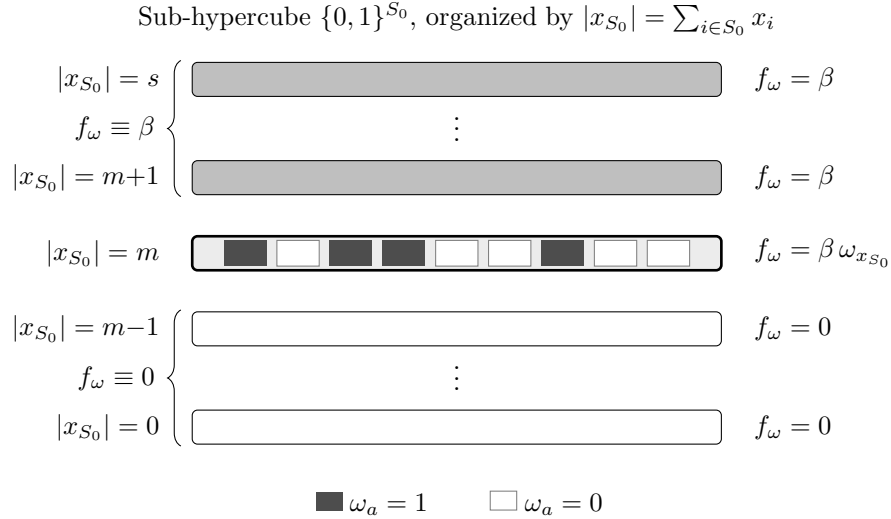

We now check that each $f_\omega$ is monotone, takes values
in~$[0,1]$, and belongs to~$\cF_K$.

\smallskip
\begin{enumerate}[label=(\roman*),nosep]
\item \emph{Monotonicity.}\; Let $x\leqslant y$ in $\Hd$.  Then
  $|x_{S_0}|\leqslant|y_{S_0}|$.  If
  $|x_{S_0}|=|y_{S_0}|$, then $x_i\leqslant y_i$ for all~$i$
  together with $\sum_{i\in S_0}x_i=\sum_{i\in S_0}y_i$ forces
  $x_i=y_i$ for every $i\in S_0$, so
  $f_\omega(x)=f_\omega(y)$.  If $|x_{S_0}|<|y_{S_0}|$,
  one of the following holds:
  \begin{itemize}[nosep]
  \item $|x_{S_0}|<m$ and $|y_{S_0}|<m$:
    $f_\omega(x)=0=f_\omega(y)$.
  \item $|x_{S_0}|<m$ and $|y_{S_0}|=m$:
    $f_\omega(x)=0\leqslant\beta\omega_{y_{S_0}}=f_\omega(y)$.
  \item $|x_{S_0}|<m$ and $|y_{S_0}|>m$:
    $f_\omega(x)=0\leqslant\beta=f_\omega(y)$.
  \item $|x_{S_0}|=m$ and $|y_{S_0}|>m$:
    $f_\omega(x)=\beta\omega_{x_{S_0}}
    \leqslant\beta=f_\omega(y)$.
  \item $|x_{S_0}|>m$ and $|y_{S_0}|>m$:
    $f_\omega(x)=\beta=f_\omega(y)$.
  \end{itemize}
  In every case, $f_\omega(x)\leqslant f_\omega(y)$.\;\checkmark

\smallskip
\item \emph{Range.}\; $f_\omega\in[0,\beta]\subseteq[0,1]$,
  provided $\beta\leqslant 1$.  Since $\beta=K/(A\sqrt{s})$, the condition $\beta\leqslant 1$
is equivalent to $K\leqslant A\sqrt{s}$.  Under the assumption
$K\leqslant c\sqrt{\log n}$ of Theorem~\ref{thm:main}(ii),
this holds by choosing~$c$ small enough (depending
on~$A$).\;\checkmark

\smallskip
\item \emph{Influence.}\; For $i\notin S_0$,
  $\Delta_i f_\omega=0$.  For $i\in S_0$,
  $\Delta_i f_\omega(x)=f_\omega(x^{i\to 1})
  -f_\omega(x^{i\to 0})$
  is nonzero only when $|x_{S_0}^{i\to 0}|$ and
  $|x_{S_0}^{i\to 1}|=|x_{S_0}^{i\to 0}|+1$ straddle the
  middle layer.  Writing $k:=|x_{S_0\setminus\{i\}}|$ (which
  determines both sides since $|x_{S_0}^{i\to 0}|=k$ and
  $|x_{S_0}^{i\to 1}|=k+1$), the two cases where $\Delta_i f_\omega\neq 0$ are:
  \begin{itemize}[nosep]
  \item $k=m-1$: $f_\omega(x^{i\to 0})=0$,
    $f_\omega(x^{i\to 1})=\beta\,\omega_{x_{S_0}^{i\to 1}}$,
    giving $\Delta=\beta\,\omega_{x_{S_0}^{i\to 1}}$.
  \item $k=m$:
    $f_\omega(x^{i\to 0})=\beta\,\omega_{x_{S_0}^{i\to 0}}$,
    $f_\omega(x^{i\to 1})=\beta$,
    giving $\Delta=\beta(1-\omega_{x_{S_0}^{i\to 0}})$.
  \end{itemize}
  Since $k=|X_{S_0\setminus\{i\}}|$ has the
  $\mathrm{Bin}(s{-}1,1/2)$ distribution and is independent
  of~$X_i$,
  \[
    I_i(f_\omega)
    = \beta\Bigl(
    \mathbb P(k{=}m{-}1)\,\E[\omega_{X_{S_0}^{i\to 1}}
    \mid k{=}m{-}1]
    +\mathbb P(k{=}m)\,\E[1{-}\omega_{X_{S_0}^{i\to 0}}
    \mid k{=}m]
    \Bigr),
  \]
  where $\mathbb P(k=j)=\binom{s-1}{j}2^{-(s-1)}$.  Each
  conditional expectation is at most~$1$.  Choosing $A>0$ large enough that
$\binom{s}{m}/2^s\leqslant A/(2\sqrt{s})$
(which is possible by Stirling's approximation), Pascal's
identity $\binom{s-1}{m-1}+\binom{s-1}{m}=\binom{s}{m}$
gives
\[
  I_i(f_\omega)
  \;\leqslant\;\frac{A\beta}{\sqrt{s}}.
\]
Summing over $i\in S_0$ and using $\beta=K/(A\sqrt{s})$,
\begin{equation}\label{eq:influence-bound}
  I(f_\omega)
  \;\leqslant\;A\beta\sqrt{s}
  \;=\; K.\;\checkmark
\end{equation}\end{enumerate}

\medskip\noindent\textbf{Step~3: Separation.}\;
For $\omega\ne\omega'\in\Omega$, the functions $f_\omega$
and~$f_{\omega'}$ differ only on the middle layer.  Since $f_\omega(x)$ depends only on~$x_{S_0}$, each term
$(f_\omega(z)-f_{\omega'}(z))^2$ is repeated $2^{d-s}$
times in the sum over~$\{0,1\}^d$, and since $f_\omega$
and~$f_{\omega'}$ agree outside~$L_m$ with
$(\omega_a-\omega_a')^2=|\omega_a-\omega_a'|$ for
$\omega_a\in\{0,1\}$,
\[
  \|f_\omega-f_{\omega'}\|_2^2
  \;=\;\frac{1}{2^s}\sum_{z\in\{0,1\}^s}
  (f_\omega(z)-f_{\omega'}(z))^2
  \;=\;\frac{\beta^2\,d_H(\omega,\omega')}{2^s}.
\]
By~\eqref{eq:VG}, $d_H(\omega,\omega')\geqslant\binom{s}{m}/4$.
Stirling's approximation gives
$\binom{s}{m}/2^s\sim\sqrt{2/(\pi s)}$ as $s\to\infty$,
so in particular $\binom{s}{m}/2^s\geqslant 4c'/\sqrt{s}$
for an absolute constant $c'>0$ and all $s\geqslant 1$.
Therefore,
\begin{equation}\label{eq:sep-final}
  \|f_\omega-f_{\omega'}\|_2^2
  \;\geqslant\;\frac{c'\,\beta^2}{\sqrt{s}}
  \;=\;\frac{c'\,K^2}{A^2\,s^{3/2}}.
\end{equation}

\subsection{Proof of Theorem~\ref{thm:main}(ii)}
\label{ssec:fano-app}

For the lower bound, it suffices to consider Gaussian noise
$\varepsilon_j\sim N(0,\sigma^2)$, which satisfies the
sub-Gaussian assumption of Section~\ref{ssec:context} and is
therefore a legitimate choice within our model.  The
KL divergence between two Gaussian location models gives
\begin{equation}\label{eq:KL}
  \KL(P_\omega\|P_{\omega'})
  \;=\; \frac{n}{2\sigma^2}\,\|f_\omega-f_{\omega'}\|_2^2
  \;\leqslant\; \frac{n\beta^2}{2\sigma^2}
  \;=\; \frac{nK^2}{2\sigma^2 A^2 s},
\end{equation}
where we used $\|f_\omega-f_{\omega'}\|_\infty\leqslant\beta$
(since both functions take values in~$[0,\beta]$).

We apply Theorem~\ref{thm:fano} with
$\delta:=\sqrt{c'}K/(2As^{3/4})$, so that
$\|f_\omega-f_{\omega'}\|_2\geqslant 2\delta$
by~\eqref{eq:sep-final}.  The Fano bound~\eqref{eq:fano} gives a positive probability of error provided
  $\bar\KL\leqslant\frac{1}{2}\log|\Omega|$ (indeed, the
  right-hand side of~\eqref{eq:fano} is then at least
  $1-(\frac{1}{2}+\log 2/\log|\Omega|)$, which is bounded
  below by a positive constant~$p_0$ since
  $|\Omega|\to\infty$ with~$n$).
  
 We now verify the condition
$\bar\KL\leqslant\frac{1}{2}\log|\Omega|$.
By~\eqref{eq:KL},
\[
  \bar\KL
  \;\leqslant\;\frac{nK^2}{2\sigma^2 A^2 s}.
\]
By~\eqref{eq:VG} and Stirling's approximation,
\[
  \frac{1}{2}\log|\Omega|
  \;\geqslant\;\frac{1}{16}\binom{s}{m}
  \;\geqslant\;\frac{c''\,2^s}{\sqrt{s}},
\]
where $c''>0$ is an absolute constant.  The condition is
therefore satisfied whenever
\begin{equation}\label{eq:fano-sufficient}
  \frac{nK^2}{2\sigma^2 A^2 s}
  \;\leqslant\;
  \frac{c''\,2^s}{\sqrt{s}}.
\end{equation}
Recall that $s=\lfloor 2\log_2 n\rfloor$, so $2^s\geqslant n^2/2$.  Under the assumption $K\leqslant c\sqrt{\log n}$,
the left-hand side of~\eqref{eq:fano-sufficient} is at most
$nc^2\log n/(2\sigma^2 A^2 s)=O(n)$. The right-hand side
is at least $c''n^2/(2\sqrt{s})=\Omega(n^2/\sqrt{\log n})$.
Since $n^2/\sqrt{\log n}\gg n$,
condition~\eqref{eq:fano-sufficient} is satisfied for $n$
large enough (depending only on $\sigma$).

\medskip\noindent\textbf{Conclusion.}\;
With $s=\lfloor 2\log_2 n\rfloor$, the Fano condition is
satisfied for $n$ large enough, and Theorem~\ref{thm:fano}
gives, for some $p_0>0$ (depending only on~$\sigma$),
\[
  \inf_{\hat f}\;\max_\omega\;
  \mathbb P_\omega\bigl(\|\hat f-f_\omega\|_2
  \geqslant\delta\bigr)
  \;\geqslant\; p_0.
\]
The construction requires $d\geqslant s$ (so that the subset
$S_0\subseteq[d]$ with $|S_0|=s$ exists); the choice of~$S_0$
is arbitrary since the uniform measure treats all coordinates
symmetrically.  Since every $f_\omega$ belongs to $\cF_K$
(by~\eqref{eq:influence-bound}), for any estimator $\hat f$,
\[
  \sup_{f\in\cF_K}R(\hat f,f)
  \;\geqslant\;\max_\omega\;
  \E_\omega\bigl[\|\hat f-f_\omega\|_2^2\bigr]
  \;\geqslant\;\max_\omega\;\delta^2\,
  \mathbb P_\omega\bigl(\|\hat f-f_\omega\|_2
  \geqslant\delta\bigr)
  \;\geqslant\;p_0\,\delta^2,
\]
where the second inequality holds since $Z^2\geqslant a^2\one_{Z\geqslant a}$
for any $Z\geqslant 0$ and $a>0$, and the third
uses the Fano bound.  Taking the infimum over~$\hat f$ and using
$s\leqslant(2/\log 2)\log n$,
\begin{equation*}
  \inf_{\hat f}\;\sup_{f\in\cF_K}\;R(\hat f,f)
  \;\geqslant\;p_0\,\delta^2
  \;=\;\frac{p_0\,c'}{4A^2}\cdot\frac{K^2}{s^{3/2}}
  \;\geqslant\;\frac{c\,K^2}{(\log n)^{3/2}},
\end{equation*}
where $c>0$ depends only on~$\sigma$.  The construction
requires $d\geqslant s=\lfloor 2\log_2 n\rfloor$, which
is guaranteed by $d\geqslant(1/c)\log n$ upon choosing
$c\leqslant\log 2/2$.  Taking~$c$ smaller if necessary,
the condition $K\leqslant c\sqrt{\log n}$ of
Theorem~\ref{thm:main}(ii) is also satisfied, completing
the proof. \qed

\subsection{Proof of Theorem~\ref{thm:adaptive}(ii)}
\label{ssec:stratified}

We adapt the construction of Section~\ref{ssec:construction}
by choosing $\beta$ to saturate $I^{(2)}\leqslant B$ rather
than $I\leqslant K$.  Let $K\leqslant c\sqrt{\log n}$,
$B\in(0,K^2/\sqrt{\log n}]$, and define
\begin{equation*}
  \beta_B := \sqrt{\frac{B}{A_1\sqrt{s}}},
\end{equation*}
where $A_1>0$ is an absolute constant chosen large enough
(depending only on~$A$) and $s=\lfloor 2\log_2 n\rfloor$
as before.  The functions $f_\omega$ are defined as in
Section~\ref{ssec:construction} with $\beta$ replaced
by~$\beta_B$; the constants $c$ and $A_1$ are chosen
small and large enough, respectively, for all conditions
below to hold.

\smallskip
\noindent\emph{Range and monotonicity.}\;
Since $s\geqslant\log n/\log 2$,
\[
  \beta_B^2 = \frac{B}{A_1\sqrt{s}}
  \leqslant \frac{K^2\sqrt{\log 2}}{A_1\log n}
  \leqslant \frac{c^2\sqrt{\log 2}}{A_1} \leqslant 1
\]
for $c$ small enough.
Monotonicity follows by the same argument as in
Section~\ref{ssec:construction}.

\smallskip
\noindent\emph{$L^1$-influence bound.}\;
By the same calculation as~\eqref{eq:influence-bound},
$I_i(f_\omega)\leqslant A\beta_B/\sqrt{s}$ for each
$i\in S_0$, so
\[
  I(f_\omega)
  \;\leqslant\; A\beta_B\sqrt{s}
  \;=\; \sqrt{\frac{A^2 B\sqrt{s}}{A_1}}
  \;\leqslant\; K\cdot\frac{A(2/\log 2)^{1/4}}{\sqrt{A_1}}
  \;\leqslant\; K,
\]
where we used $\sqrt{s}\leqslant\sqrt{2/\log 2}\sqrt{\log n}$
and $B\leqslant K^2/\sqrt{\log n}$, and the last inequality
holds by choosing $A_1\geqslant A^2\sqrt{2/\log 2}$.

\smallskip
\noindent\emph{$L^2$-influence bound.}\;
Since $\Delta_i f_\omega\in\{0,\beta_B\}$, we have
$(\Delta_i f_\omega)^2=\beta_B\,\Delta_i f_\omega$ pointwise,
hence $I_i^{(2)}(f_\omega)=\beta_B\,I_i(f_\omega)$ for each
$i\in S_0$.  Summing over $i\in S_0$,
\[
  I^{(2)}(f_\omega)
  \;=\;\beta_B\,I(f_\omega)
  \;\leqslant\;A\beta_B^2\sqrt{s}
  \;=\;\frac{AB}{A_1}
  \;\leqslant\; B,
\]
using $I(f_\omega)\leqslant A\beta_B\sqrt{s}$ from the $L^1$-bound
above and $A_1\geqslant A$.
Hence $f_\omega\in\cF_{K,B}$.

\smallskip
\noindent\emph{Separation.}\;
By the same calculation as~\eqref{eq:sep-final},
\[
  \|f_\omega-f_{\omega'}\|_2^2
  \;\geqslant\;\frac{c'\beta_B^2}{\sqrt{s}}
  \;=\;\frac{c'B}{A_1 s},
\]
where $c'>0$ is an absolute constant.

\smallskip
\noindent\emph{Fano condition.}\;
$\bar\KL\leqslant n\beta_B^2/(2\sigma^2)
=nB/(2\sigma^2 A_1\sqrt{s})$.
Using $B\leqslant K^2/\sqrt{\log n}\leqslant c^2\sqrt{\log n}$
and $\sqrt{s}\geqslant\sqrt{\log n/\log 2}$,
\[
  \bar\KL\;\leqslant\;\frac{c^2 n\sqrt{\log 2}}{2\sigma^2 A_1}
  \;=\;O(n)
  \;\ll\;\tfrac{1}{2}\log|\Omega|
  \;=\;\Omega\!\left(\frac{n^2}{\sqrt{\log n}}\right).
\]
The Fano condition is therefore satisfied for $n$ large enough.

\smallskip
\noindent\emph{Conclusion.}\;
Applying Theorem~\ref{thm:fano} with
$\delta:=\sqrt{c'B/(4A_1 s)}$ and using
$s\leqslant(2/\log 2)\log n$,
\[
  \inf_{\hat f}\;\sup_{f\in\cF_{K,B}}\;R(\hat f,f)
  \;\geqslant\; p_0\,\delta^2
  \;=\;\frac{p_0 c'B}{4A_1 s}
  \;\geqslant\;\frac{cB}{\log n},
\]
which is~\eqref{eq:adaptive-lower}, adjusting~$c$.\qed

\section{Discussion}\label{sec:discussion}

Theorem~\ref{thm:main} leaves a gap between the upper bound
$O(K/\sqrt{\log n})$ and the lower bound
$\Omega(K^2/(\log n)^{3/2})$.  For fixed~$K$, this gap is a
factor of~$\log n$.  However, Theorem~\ref{thm:adaptive} shows
that the picture is sharper than it appears: stratifying $\cF_K$
by $I^{(2)}(f)$, the gap on each sub-class~$\cF_{K,B}$ reduces
to~$\sqrt{\log n}$, with matching bounds $CB/\sqrt{\log n}$ (upper)
and $cB/\log n$ (lower).  The residual $\log n$ gap on the full class~$\cF_K$ is an
aggregation artifact: the upper bound is largest when
$I^{(2)}(f)\approx K$ (near-Boolean functions, for which
$\Delta_i f\in\{0,1\}$ and $I^{(2)}=I$), whereas the lower
bound construction achieves $I^{(2)}(f_\omega)\asymp
K^2/\sqrt{\log n}$. We discuss
the sources of this remaining gap and several directions
for further work.

\medskip\noindent\textbf{The upper bound bottleneck.}\;
The rate $K/\sqrt{\log n}$ arises from a bias-variance
trade-off in the degree parameter~$d_0$: the bias decreases
as $d_0$ grows, but the estimated spectral set~$\hat\cS$ has cardinality
at most $e^{Cd_0^2}$ for a constant $C$ depending on~$\gamma$
(see~\eqref{eq:log-Shat}), forcing
$d_0\lesssim\sqrt{\log n}$ to keep the variance bounded.

The root cause is that our
estimator includes all subsets of~$\hat J$ of
size~$\leqslant d_0$ in~$\hat\cS$, many of which carry
negligible Fourier weight.  Identifying from data the
relevant Fourier subsets~$\cS\subseteq 2^J$ rather than
all subsets of~$\hat J$ would allow a larger~$d_0$ and
a faster rate, but this higher-order selection problem
remains open.

\begin{openproblem}\label{op:selection}
  Is there an estimator that identifies the relevant
  Fourier subsets~$\cS$ (not just the influential
  coordinates~$J$) from noisy observations, with
  estimation cost proportional to~$|\cS|$ rather
  than~$|\hat\cS|$?
\end{openproblem}

\medskip\noindent\textbf{The lower bound bottleneck.}\;
The rate $K^2/(\log n)^{3/2}$ reflects the geometry of the
middle layer of $\{0,1\}^s$ ($s=\Theta(\log n)$): its
$\mu$-measure $\Theta(1/\sqrt{s})$ forces the scaling
$\beta=\Theta(K/\sqrt{s})$ to fit the influence budget,
yielding a squared separation of order $K^2/s^{3/2}$.
Improving the lower bound would require either packing
functions that vary on a larger fraction of the hypercube,
or a testing argument that goes beyond Fano's inequality.

\begin{openproblem}\label{op:rate}
  What is the exact minimax rate of estimation over
  $\cF_{K,B}$?  Can the $\sqrt{\log n}$ gap between
  $CB/\sqrt{\log n}$ and $cB/\log n$ in
  Theorem~\ref{thm:adaptive} be closed?
\end{openproblem}

Our conjecture is that the upper bound $CB/\sqrt{\log n}$
is the correct rate, and that the lower bound construction
does not capture the full difficulty of~$\cF_{K,B}$.

\medskip\noindent\textbf{Comparison with isotonic regression on $[0,1]^d$.}\;
For isotonic regression on~$[0,1]^d$, \citet{han2019} showed
that the minimax rate is $n^{-1/d}$ (up to logarithmic
factors), which becomes uninformative for $d\gtrsim\log n$.  On the
Boolean hypercube, the influence constraint rescues estimation
even for $d$ nearly exponential in~$n$: the rate
$K/\sqrt{\log n}$ remains finite under the mild condition
$\log d\leqslant n^{1-\varepsilon}$.  The two settings are not
directly comparable (different domains, different structural
assumptions), but the qualitative message is the same:
meaningful high-dimensional estimation requires constraints
beyond monotonicity.  Whether an analogue of the influence
budget can be defined on~$[0,1]^d$ is an interesting open
question.

\medskip\noindent\textbf{Computational cost.}\;
Step~1 of the estimator (influence estimation) requires $O(nd)$
operations.  Step~2 (Fourier estimation) computes one average
per element of~$\hat\cS$, at cost $O(n|\hat\cS|)$.  Since
$\log|\hat\cS|=O(\log n)$, the total cost is $O(n^{c})$ for a
constant~$c$ depending on~$\gamma$---polynomial in~$n$, but
with a potentially large exponent.  In practice, one may choose
a smaller~$d_0$ (e.g., $d_0=2$ or $3$), accepting a larger
bias in exchange for a much smaller computational cost.

\medskip\noindent\textbf{Extensions.}\;
Our results are stated for the uniform measure $\mu$ on~$\Hd$.
The Fourier expansion and Bonami--Beckner inequality extend to
general product measures
$\mu_1\otimes\cdots\otimes\mu_d$ on $\{0,1\}^d$
\citep{odonnell2014}, and the spectral concentration argument
carries over with modified constants when the marginal
probabilities $\mu_i(\{1\})$ are bounded away from $0$ and~$1$.
The degenerate case (some marginals close to $0$ or~$1$)
introduces additional difficulties and is left for future work.

One might also consider the class
$\{f\in\cM_d:I^{(2)}(f)\leqslant K\}$, constraining the
$L^2$-influence directly.  The spectral concentration still
applies (since $I^{(2)}\leqslant I$, this is a larger class),
but the estimation of~$I_i^{(2)}(f)=\E[(\Delta_i f(X))^2]$
from data is harder: unlike the $L^1$-influence, it is not
a simple difference of means.  Whether the minimax rate
changes under this alternative constraint is an open question.

\bibliographystyle{apalike}
\bibliography{references}

\end{document}